\documentclass[10pt,smallextended,referee,envcountsect]{article}
\usepackage[utf8]{inputenc}
\usepackage[T1]{fontenc}
\usepackage{amsmath}
\usepackage{tikz}
\usepackage[french,english]{babel}
\usepackage{amsfonts}
\usepackage{xfrac}
\usepackage{amssymb}
\usepackage{mathrsfs}
\usepackage{dsfont}
\usepackage[numbers]{natbib}
\usepackage{mathabx}
\usepackage{bbold}
\usepackage{amsthm}
\usepackage{indentfirst}

\newtheorem{prop}{Proposition}[section]
\newtheorem{lemma}{Lemma}[section]
\newtheorem{theorem}{Theorem}[section]

\usepackage{color}

\newcommand{\bs}{\boldsymbol}
\newcommand{\ldb}{\ldbrack}
\newcommand{\rdb}{\rdbrack}
\newcommand{\N}{\mathbb{N}}
\newcommand{\R}{\mathbb{R}}
\newtheoremstyle{Roman}
  {\topsep}
  {\topsep}
  {\normalfont}
  {0pt}
  {\bfseries}
  {.}
  { }
  {\thmname{#1}\thmnumber{ #2}\thmnote{ (#3)}}
\theoremstyle{Roman}
\newtheorem{mydef}{Definition}[section]
\newtheorem{rem}{Remark}[section]
\newtheorem{ex}{Example}[section]

\author{Xavier Bacon \footnote{Statistique, Analyse et Modélisation 
Multidisciplinaire (SAMM), Centre Pierre Mendès France 90, rue de Tolbiac
75634 Paris cedex 13, France.} \footnote{E-mail adress: xavier.bacon@etu.univ-paris1.fr}}
\title{\textbf{OPTIMAL TRANSPORTATION OF VECTOR-VALUED MEASURES}}
\begin{document}

\date{}
\maketitle

\begin{abstract}
Given two $n$-dimensional measures $\bs \mu$ and $\bs \nu$ on Polish spaces, we propose an optimal transportation's formulation, inspired by classical Kantorovitch's formulation in the scalar case. In particular, we established a strong duality result and as a consequence, optimality conditions are investigated. Wasserstein's metrics induced by our formulation are also investigated.
\end{abstract}


\section{Introduction and notations}
\subsection{Introduction}

Starting from the article of G. Monge \cite{monge1781memoire}, many mathematical formulations of optimal transportation have been offered (\cite{kantorovitch1958translocation}, \citep{beckmann1952continuous} and \cite{benamou2000computational}). In Monge's formulation, given two Polish spaces $X$ and $Y$, if $\mu$ (resp. $\nu$) is a Borelian probability on $X$ (resp. $Y$) and if $ c : X \times Y \rightarrow \mathbb{R}$, then the Monge's formulation consists on a minimization of the total cost among all Borelian maps which push forward $\mu$ to $\nu$, more precisely for $T$ a Borelian function between $X$ and $Y$ and $m$ a positive measure on $X$, $T \# \mu$ stands for the push forward measure which is the measure on $Y$ defined for all measurable set $B$ by $T \# \mu (B) := \mu \left[ T^{-1} (B) \right]$. Let $M(\mu,\nu)$ be the set of such maps, Monge transportation problem is then
\begin{align}
\mathcal{M}(\mu,\nu) := \inf \left\{ \int_X c \left[ x,T(x) \right] \, \mathrm{d} \mu(x) : T \in M(\mu,\nu) \right\}. \label{Monge}
\end{align}

In the middle of the 20th century, L. Kantorovitch proposed a relaxation of (\ref{Monge}) in \cite{kantorovitch1958translocation} by allowing mass splitting. Thinking of $\mu$ and $\nu$ as piles of sands, grains located at $x$ can be sent at different places at the same time. Formally, Kantorovitch's problem consists on minimizing a new total cost among all transference plans $\gamma \in \Pi(\mu,\nu)$, where $\Pi(\mu,\nu)$ is the set of couplings between $\mu$ and $\nu$ $i.e.$ $\gamma \in \Pi(\mu,\nu)$ if for all $A$ Borelian subset of $X$, $\gamma(A \times Y) = \mu(A)$ and for all $B$ Borelian subset of $Y$, $\gamma(X \times B) = \nu(B)$. Kantorovitch's transportation problem is then
\begin{align}
\mathcal{K}(\mu,\nu) := \inf \left\{ \iint_{X \times X} c(x,y) \, \mathrm{d} \gamma (x,y) : \gamma \in \Pi(\mu,\nu) \right\} \label{Kanto}
\end{align}

\noindent and for reasons that are discussed below, (\ref{Kanto}) is more accurate to extend the classical theory to vector-valued measures.

When $c$ is the power of a distance, these two problems induce a metric on the set of probabilities, called here Wasserstein metric (see \citep{ambrosio2008gradient}, \cite{villani2003topics}, \cite{santambrogio2015optimal} or \cite{villani2008optimal}). In the recent years, extensions of optimal transportation to more general objects have been proposed, such as multimarginal transportation (\cite{carlier2008optimal},\cite{kitagawa2015multi}) or density functional theory (\citep{cotar2015infinite}). Notice also that optimal transportation of matricial and tensorial measures (see \cite{carlen2014analog}, \citep{chen2017matricial}) or vector-valued densities in \cite{zinsl2015transport} have already been investigated.

In the present paper, we propose an extension to vector-valued measures. This one is deeply based on Kantorovich's formulation of $scalar$ optimal transportation (section \ref{KP}). Given two probabilities $\mu$ and $\nu$ and two decompositions of them (say) $\mu = \mu_1 + \cdots + \mu_n$ and $\nu = \nu_1 + \cdots + \nu_n$, more than a transportation between $\mu$ and $\nu$, we are interested in a description of a transportation between these two decompositions. A naive strategy would be to study the $n$ subproblems of classical optimal transportation between $\mu_i$ and $\nu_i$ for $i \in \ldbrack 1,n \rdbrack$, assuming that for all $i,\mu_i$ and $\nu_i$ share the same mass. If this new transportation problem leads to a metric, then the toplogy induced is the product one, due to the independance of each $phasis$. This problem has been explored in \cite{benamou2000numerical},\cite{benamou2004numerical} and more recently in \cite{lavenant2017time}. To remove the independance of each phasis, we allow transformation similarly as explored in \cite{chen2018vector}. Introducing $n^2$ particular costs $c_{ij}$ and $n^2$ particular transference plans $\gamma_{ij}$ which describe the transport of a piece of $\mu_i$ into a piece of $\nu_j$, we consider that the cost transportation to move $\mathrm{d} \mu_i(x)$ to $\mathrm{d} \nu_j(y)$ is $c_{ij}(x,y) \mathrm{d} \gamma_{ij} (x,y)$. Compatibility constraints are described by the set $\bs \Pi(\bs \mu,\bs \nu)$ where we ask that the $n^2$ transport plans $\gamma_{ij}$ clear each $\mu_i$ and fill each $\nu_j$. The new minimization problem is given by
\begin{align*}
\inf \left\{ \sum \limits_{(i,j) \in \ldb 1,n \rdb^2} \iint_{X \times Y} c_{ij}(x,y) \, \mathrm{d} \gamma_{ij}(x,y), \boldsymbol{\gamma} \in \bs \Pi \left( \begin{bmatrix}
\mu_1 \\
\vdots \\
\mu_n
\end{bmatrix}, \begin{bmatrix}
\nu_1 \\
\vdots \\
\nu_n
\end{bmatrix} \right) \right\}.
\end{align*}

In section \ref{KP}, we give an existence result for this problem as well as various examples. Then, following the shipper's problem interpretation of optimal transportation from L. Caffarelli (presented in \cite{villani2003topics}), we introduce a dual formulation in section \ref{DP} and prove strong duality theorem. As a consequence of the duality, optimality conditions for primal-dual optimizers are derived. Finally, assuming that costs $(c_{ij})$ are all the same power of different distances, a metric on vector-valued measures is presented in section \ref{Topology}.

\subsection{Notations}

In this article, we differenciate vectorial objects from scalar ones by using bold type character like $\bs \Pi$ for the first one and non-bold type character like $\Pi$ for the latter one.
\begin{itemize}
\item Given $X$ a measurable space, $\mathcal{P}(X)$ stands for the set of probability measures on $X$ and for all $n \in \mathbb{N}^*$, $\mathcal{M}^n(X)$ (resp. $\mathcal{M}^{n \times n}(X)$) refers to the set of vectorial measure on $(X,\mathcal{X})$ valued in $\mathbb{R}^n$ (resp. in $\mathbb{R}^{n \times n}$) meaning that each coordinate is a signed measure. $\mathcal{M}_+^n(X)$ (resp. $\mathcal{M}_+^{n \times n}(X)$) stands for the subset where each coordinate is positive measure.  Recall that for $T$ a measurable function between $X$ and $Y$ and $m$ a positive measure on $X$, $T \# m$ stands for the push forward measure which is the measure on $Y$ defined for all measurable set $B$ by $T \# m (B) = m \left[ T^{-1} (B) \right]$.

\item Given $X_1 \times \cdots \times X_n$ a product space and $k \in \ldb 1,n \rdb$, $\pi_k$ denotes the canonical projection on $X_k$ $i.e$,
\[
\begin{array}{lllll}
\pi_k & : & X_1 \times \cdots \times X_n & \longrightarrow & X_k \\
& & (x_1, \cdots, x_n) & \mapsto & x_k 
\end{array}
\]
\noindent and for $l \in \ldb 1,n \rdb$ and $l > k$, $\pi_{k,l}$ denotes the canonical projection on $X_k \times X_l$ $i.e$,
\[
\begin{array}{lllll}
\pi_{k,l} & : & X_1 \times \cdots \times X_n & \longrightarrow & X_k \times X_l \\
& & (x_1, \cdots, x_n) & \mapsto & (x_k,x_l) 
\end{array}
\]

\item For $A$ a borelian subset of $\mathbb{R}$, $\mathcal{L}_A$ stands for the Lebesgue measure on $A$. If $m,M \in \mathcal{M}_+(X)$ satisfy for all $A \in \mathcal{X}, m(A) \leqslant M(A)$, $m$ is called a submeasure of $M$ and this property will be written $m \leqslant M$. Note that being a submeasure of $M$ implies the absolute continuity w.r.t. $M$.

\item Given $(X,\mathcal{T})$ a topological space and $(Y,d)$ a metric space, $C_b(X,Y)$ refers to the set of bounded continuous functions between $(X,\mathcal{T})$ and $(Y,d)$.

\item Given $(i,j) \in \ldb 1,n \rdb^2$, $\bs{E_{ij}}$ refers to the matrix $n \times n$ whose coordinates are all equal to $0$ except $(i,j)$ which is equal to $1$.

\item Given a set $X$ and $S$ a subset of $X$, $\iota_{S}$ denotes for the function equals to $0$ on $S$ and $+ \infty$ on its complementary.

\item The notation $\wedge$ will be used to denote the minimum of two reals, and $\vee$ for the maximum.
\end{itemize}

\section{Kantorovitch's problem} \label{KP}
\subsection{Presentation}
In the remainder of the paper, $n$ will denote an element of $\mathbb{N}^*$.

\begin{mydef}
Given $(X,\mathcal{X})$ a mesurable space, $\mathcal{P}^n(X)$ denotes the set of admissible distributions of $n$ species defined by
\begin{align*}
\mathcal{P}^n(X) = \left\{ \textbf{m} =\begin{bmatrix}
m_1 \\
\vdots \\
m_n
\end{bmatrix} \in \mathcal{M}^n_+(X) : \sum \limits_{i=1}^n m_i \in \mathcal{P}(X) \right\}.
\end{align*}
\end{mydef}

\noindent It is straightforward that $\mathcal{P}^n(X)$ is a non-empty convex subset of $\mathcal{M}^n(X)$.

Inspired by Kantorovitch's formulation of optimal transportation, an extension of the notion of transference plan between two scalar measures is now proposed. For a well understanding of the next definition, let us make a short digression and present our model. Given $(X, \mathcal{X})$ and $(Y, \mathcal{Y})$ two measurable spaces and $\bs \mu \in \mathcal{P}^n(X), \bs \nu \in \mathcal{P}^n(Y)$ two distributions of $n$ species, since the total amount of each specy is not equal transformations between species are allowed. Given $(i,j) \in \ldbrack 1,n \rdbrack^2$, the ''transportation'' (with ''transformation'' if $i \neq j$) of a piece of $\mu_i$ into a piece of $\nu_j$ is described by a transference plan $\gamma_{ij} \in \mathcal{M}_+(X \times Y)$. Constraints on $\bs \gamma = (\gamma_{ij})_{1 \leqslant i,j \leqslant n}$ are given by
\begin{align*}
(\mbox{Clear } \bs \mu) \,\, \forall i \in \ldbrack 1,n \rdbrack , \forall A \in \mathcal{X}, \mu_i(A) &= \sum \limits_{k=1}^{n} \gamma_{ik}(A \times Y) \\
(\mbox{Fill } \bs \nu) \,\, \forall j \in \ldbrack 1,n \rdbrack, \forall B \in \mathcal{Y}, \nu_j(B) &= \sum \limits_{k=1}^{n} \gamma_{kj}(X \times B),
\end{align*}

\noindent or in other words, for all $(i,j) \in \ldb 1,n \rdb^2, \sum_{k=1}^{n} \gamma_{ik}$ has $\mu_i$ as first marginal and $\sum_{k=1}^{n} \gamma_{kj}$ has $\nu_j$ as second marginal. This naturally leads to the following definition.

\begin{mydef} \label{def transf plan}
Given $\boldsymbol{\mu} \in \mathcal{P}^n(X)$ and $\boldsymbol{\nu} \in \mathcal{P}^n(Y)$, $\bs \Pi(\boldsymbol{\mu},\boldsymbol{\nu})$ denotes the set of transference plans between $\boldsymbol{\mu}$ and $\boldsymbol{\nu}$ defined by
\begin{align*}
\bs \Pi(\boldsymbol{\mu},\boldsymbol{\nu}) = \left\{ \begin{array}{cc}
& \mu_i = \pi_1 \# \left( \sum \limits_{k=1}^{n} \gamma_{ik} \right) \\
{\bs \gamma} \in \mathcal{M}^{n \times n}_+(X \times Y) : \forall (i,j) \in \ldb 1,n \rdb^2, & \\
& \nu_j = \pi_2 \# \left( \sum \limits_{k=1}^{n} \gamma_{kj} \right)
\end{array}
\right\}.
\end{align*}
\end{mydef}

\begin{rem} According to the Definition \ref{def transf plan}, every $\bs \gamma \in \bs \Pi (\bs \mu, \bs \nu)$ induces a canonical transference plan (for $n=1$, the two definitions of transference plan are the same) between $\sum_{i=1}^{n} \mu_i$ and $\sum_{j=1}^{n} \nu_j$ given by $\sum_{i,j = 1}^{n} \gamma_{ij}$. However the converse is not true since given $\gamma \in \Pi \left( \sum_{i=1}^{n} \mu_i, \sum_{j=1}^{n} \nu_j \right)$ and $(x,y) \in X \times Y$, there is still a choice to make: is the first specy sent into the first or the second one or both? And in what proportions? Let us give a short example to clarify this remark. Taking
\begin{align*}
\bs \mu = \begin{bmatrix}
\mu_1 \\
\mu_2
\end{bmatrix} = \frac{1}{2} \begin{bmatrix}
 \mathcal{L}_{[-1,0]} \\
 \mathcal{L}_{[-1,0]}
\end{bmatrix}, \bs \nu = \begin{bmatrix}
\nu_1 \\
\nu_2
\end{bmatrix} = \frac{1}{2} \begin{bmatrix}
 \mathcal{L}_{[0,1]} \\
 \mathcal{L}_{[0,1]}
\end{bmatrix}
\end{align*}
and writting $\tau_1 : x \rightarrow x + 1$, it is known that $\gamma = \left( I,\tau_{1} \right) \# \left( \mu_1 + \mu_2 \right)$ is a transference plan between $\mathcal{L}_{[-1,0]} = \mu_1 + \mu_2$ and $\mathcal{L}_{[0,1]} = \nu_1 + \nu_2$. Given a such $\gamma$, $\mu_1$ can be sent towards $\nu_1$, or towards $\nu_2$. A mix is even possible and $\mu_1$ can be sent towards $\mathds{1}_{[0,\frac{1}{2}]} \, \mathrm{d} \nu_1 + \mathds{1}_{[\frac{1}{2},1]} \, \mathrm{d} \nu_2$. In other words, the following matrix measures are transference plans,
\begin{align*}
\begin{bmatrix}
(I,\tau_{1}) \# \mu_1 & 0 \\
0 & (I,\tau_{1}) \# \mu_2
\end{bmatrix} , \\
\begin{bmatrix}
0 & (I,\tau_{1}) \# \mu_1 \\
(I,\tau_{1}) \# \mu_2 & 0
\end{bmatrix}, \\
\frac{1}{2} \begin{bmatrix}
(I,\tau_{1}) \# \mu_1 & (I,\tau_{1}) \# \mu_1 \\
(I,\tau_{1})\# \mu_2 & (I,\tau_{1}) \# \mu_2
\end{bmatrix}.
\end{align*}
\end{rem}

We also introduce matrix-valued cost $\textbf{c}$ as a function from $X \times Y \rightarrow \mathcal{M}_n(\mathbb{R})$, integrable w.r.t. $\bs \gamma$ or positive measurable. The associated total cost is given by the following definition.

\begin{mydef}
Given $\boldsymbol{\gamma} \in \bs \Pi(\boldsymbol{\mu},\boldsymbol{\nu})$ and a cost matrix $\mathbf{c}$, $K( \boldsymbol{\gamma})$ denotes the total transportation cost according to $\bs \gamma$ defined by
\begin{align*}
K(\boldsymbol{\gamma}) = \sum \limits_{(i,j) \in \ldb 1,n \rdb^2} \iint_{X \times Y} c_{ij}(x,y) \, \mathrm{d} \gamma_{ij}(x,y).
\end{align*}
\end{mydef}

\noindent The Kantorovich's transportation problem between two distributions of $n$ species $\boldsymbol{\mu}$ and $\boldsymbol{\nu}$ for $\mathbf{c}$ is given by 
\begin{align*}
\tag{KP} \inf \left\{ K(\boldsymbol{\gamma}) : \boldsymbol{\gamma} \in  \bs \Pi(\boldsymbol{\mu},\boldsymbol{\nu}) \right\} =: \mathcal{K}(\boldsymbol{\mu},\boldsymbol{\nu}) \in [- \infty, + \infty]
\end{align*}

\begin{ex} \label{basic example}
Note that if $c_{ij} = c$ for all $(i,j) \in \ldb 1,n \rdb^2$ then (KP) shares the same value as the $scalar$ optimal transportation between $\sum_{i=1}^{n} \mu_i$ and $\sum_{j=1}^{n} \nu_j$ for the cost $c$. The most simple example of non trivial matrix cost is given by the following one: let $c$ be a $scalar$ cost and $\kappa$ be a real and define the following matrix cost:

\begin{align*}
\forall (x,y) \in X \times Y, \bs c(x,y) = \begin{bmatrix}
c(x,y) & c(x,y) + \kappa \\
c(x,y) + \kappa & c(x,y)
\end{bmatrix}.
\end{align*}

\noindent In other words, a constant cost is requiered for any transformation. See the example \ref{big example} below for a study of this special cost.
\end{ex}

\begin{ex} \label{Dirac example}
Let $ \bs p, \bs q \in \Delta_n := \left\{ \bs x \in \R_+^n, \sum_{i=1}^n x_i = 1 \right\}$ and $ \bs x, \bs y \in X^n$. Define
\begin{align*}
\bs \mu =\begin{bmatrix}
p_1 \delta_{x_1} \\
\vdots \\
p_n \delta_{x_n}
\end{bmatrix}, \bs \nu = \begin{bmatrix}
q_1 \delta_{y_1} \\
\vdots \\
q_n \delta_{y_n}
\end{bmatrix}.
\end{align*}

\noindent Let us first notice that for all $(i,j)$, since $\mathrm{supp}(\sum_{l=1}^{n}{\gamma_{il}}) \subseteq \mathrm{supp}(\mu_i) \times Y$ and $\mathrm{supp}(\sum_{l=1}^{n}{\gamma_{lj}}) \subseteq X \times \mathrm{supp}(\nu_j)$ then $\mathrm{supp}(\gamma_{ij}) \subseteq {(x_i,y_j)}$ and hence $\gamma_{ij} = t_{ij} \delta_{(x_i,y_j)}$ for some $t_{ij} \in [0,1]$. Constraints on $\bs \gamma$ give us that for all  $(i,j) \in \ldb 1,n \rdb^2, \sum_{\substack{l=1}}^{n}{t_{il}} = p_i$ and $\sum_{\substack{l=1}}^{n}{t_{lj}} = q_j$ and finally
\begin{align*}
\bs K(\bs \gamma) = \sum_{(i,j) \in \ldb 1,n \rdb^2} t_{ij} c_{ij}(x_i,y_j),
\end{align*}
\noindent(KP) becomes in that case
\begin{align*}
\inf \left\{ \sum_{(i,j) \in \ldb 1,n \rdb^2} t_{ij} c_{ij}(x_i,y_j), \bs t \in \mathcal{M}_n(\mathbb{R}) : \sum\limits_{\substack{l=1}}^{n}{t_{il}} = p_i, \sum\limits_{\substack{l=1}}^{n}{t_{lj}} = q_j   \right\}
\end{align*}
\noindent which reduces to the discret optimal transportation.
\end{ex}

\subsection{Existence of a minimizer}

Let $X$ and $Y$ be two Polish spaces. In this subsection, we prove an existence result for the problem (KP). Arguments used to establish it are the same as in scalar case (see \cite{villani2003topics} or \cite{santambrogio2015optimal} for instance). Let us first gather the main structural properties of problem (KP).
\begin{lemma} \label{lemmaexistenceprimal} Given $\boldsymbol{\mu} \in \mathcal{P}^n(X)$, $\boldsymbol{\nu} \in \mathcal{P}^n(Y)$ and $\boldsymbol{c}$ a cost matrix, following assertions are satisfied:

\textbf{[1]} $\bs \Pi(\boldsymbol{\mu},\boldsymbol{\nu})$ is a non-empty convex subset of $\mathcal{M}^{n \times n}(X \times Y)$.

\textbf{[2]} $\bs \Pi(\boldsymbol{\mu},\boldsymbol{\nu})$ is a weakly sequentially compact \footnote{w.r.t. the test function space $C_b(X \times Y, \mathbb{R}^{n \times n})$.} subset of $\mathcal{M}^{n \times n}(X \times Y)$.

\textbf{[3]} If for all $(i,j) \in \ldbrack 1,n \rdbrack^2$, $c_{ij}$ is bounded from below, then $K : \bs \Pi(\boldsymbol{\mu},\boldsymbol{\nu}) \rightarrow \mathbb{R}\cup \{ + \infty \}$ is bounded from below.

\textbf{[4]} If for all $(i,j) \in \ldbrack 1,n \rdbrack^2$, $c_{ij}$ is l.s.c. and bounded from below then $K : \bs \Pi(\boldsymbol{\mu},\boldsymbol{\nu}) \rightarrow \mathbb{R}\cup \{ + \infty \}$ is weakly l.s.c. with respect to the tight convergence.
\end{lemma}

\begin{proof}

[1] Convexity is clear and it is easy to check that $\left( \mu_i \otimes \nu_j \right)_{(i,j) \in \ldb 1,n \rdb^2} \in \bs \Pi(\boldsymbol{\mu},\boldsymbol{\nu})$.

 \noindent [2] Let $(\boldsymbol{\gamma^k})_{k \in \mathbb{N}} \in \bs \Pi(\boldsymbol{\mu},\boldsymbol{\nu})^{\mathbb{N}}$ and $(i,j) \in \ldbrack 1,n \rdbrack^2$. We claim that $(\gamma_{ij}^k)_{k \in \mathbb{N}}$ is tight. Indeed, let $\varepsilon \in \mathbb{R}^*_+$ and $K_X$ (resp. $K_Y)$ a compact of $X$ (resp. $Y$) such that\footnote{These two compacts exist: all $\mu_i$ are finite measure on Polish spaces then it exists $K^i_X$ verifying these inequalities and then we just have to take union of them.}
\begin{align}
\forall i \in \ldbrack 1,n \rdbrack, \mu_i(X \setminus K_X) \leq \varepsilon \mbox{ (resp. } \forall j \in \ldbrack 1,n \rdbrack, \nu_j(Y \setminus K_Y) \leq \varepsilon\mbox{)}.
\end{align}

\noindent Let $k \in \mathbb{N}$, following inequalities are satisfied,
\begin{align*}
\gamma^k_{ij}\left[(X \times Y) \backslash (K_X \times K_Y)\right] &\leqslant \gamma_{ij}^k[(X \backslash K_X) \times Y)] + \gamma_{ij}^k[X \times (Y \backslash K_Y)] \\
&\leqslant \sum \limits_{l=1}^{n} \gamma_{il}^k[(X \backslash K_X) \times Y)] + \sum \limits_{\tilde{l}=1}^{n} \gamma_{\tilde{l}j}^k[X \times (Y \backslash K_Y)]  \\
&= \mu_i(X \backslash K_X) + \nu_j(Y \backslash K_Y) \mbox{ since $\bs \gamma \in \bs \Pi(\bs \mu,\bs \nu)$} \\
&\leqslant 2 \varepsilon.
\end{align*}
This proves the claim and thanks to Prokhorov theorem, there exists a non-negative finite measure on $X \times Y, \gamma_{ij}^{\infty}$ and a subsequence of $(\gamma^k_{ij})_{k \in \mathbb{N}}$ (still written $(\gamma^k_{ij})_{k \in \mathbb{N}}$) such as $(\gamma^k_{ij})_{k \in \mathbb{N}}$ tightly converges towards $\gamma_{ij}^{\infty}$. In order to conclude, we only have to check that $\boldsymbol{\gamma}^{\infty} \in \Gamma(\boldsymbol{\mu},\boldsymbol{\nu})$ Let $\phi \in C_b(X \times Y, \mathbb{R})$ and notice that for all $i \in \ldbrack 1, n \rdbrack$ and $k \in \mathbb{N}$,
\begin{align*}
\int_X \phi(x) \, \mathrm{d} \mu_i(x) = \sum \limits_{l=1}^{n} \iint_{X \times Y} \phi(x) \, \mathrm{d}\gamma^k_{il}(x,y) \rightarrow \sum \limits_{l=1}^{n} \iint_{X \times Y} \phi(x) \, \mathrm{d}\gamma_{il}^{\infty}(x,y)
\end{align*}.

\noindent [3] Straightforward.

\noindent [4] Let $(\boldsymbol{\gamma}^k)_{k \in \mathbb{N}} \in \Pi(\boldsymbol{\mu},\boldsymbol{\nu})^{\mathbb{N}}$ and $\boldsymbol{\gamma^{\infty}} \in \Pi(\boldsymbol{\mu},\boldsymbol{\nu})$ such that $(\boldsymbol{\gamma^k})_{k \in \mathbb{N}}$ tightly converges towards $\boldsymbol{\gamma^{\infty}}$ in that for all $(i,j) \in \ldbrack 1,n \rdbrack^2, (\gamma^k_{ij})_{k \in\mathbb{N}}$ weakly converges in duality with $C_b$ towards $\gamma^{\infty}_{ij}$. Then, by lower semi-continuity of $\gamma_{ij} \mapsto <\gamma_{ij},c_{ij}>$ (see \cite{santambrogio2015optimal}, Lemma 1.6), for all $(i,j) \in \ldbrack 1,\cdots,n \rdbrack^2$,
\begin{align*}
\iint_{X \times Y} c_{ij}(x,y) \, \mathrm{d} \gamma_{ij}^{\infty}(x,y) \leqslant \underset{k \rightarrow \infty}{\liminf} \, \iint_{X \times Y} c_{ij}(x,y) \, \mathrm{d} \gamma^k_{ij}(x,y)
\end{align*}
and since sum of $\liminf$ is less or equal to $\liminf$ of sum, it ends the proof.
\end{proof}

\noindent With these facts in hand, our main result easily follows.

\begin{theorem} \label{Existence th in K}
Given $\mathbf{c}$ a cost matrix such as for all $(i,j) \in \ldbrack 1,n \rdbrack^2, c_{ij}$ is bounded from below and $l.s.c.$, it exists $\boldsymbol{\gamma} \in \bs \Pi(\boldsymbol{\mu},\boldsymbol{\nu})$ such as $K(\boldsymbol{\gamma})= \mathcal{K}(\boldsymbol{\mu},\boldsymbol{\nu})$.
\end{theorem}

\begin{proof}
This proof follows the classical direct method of calculus of variations. Let $(\bs \gamma^k)_{k \in \N}$ be a minimizing sequence for the problem $(KP)$ $i.e$
\begin{align*}
\forall k \in \N, K(\bs \gamma^k) \leqslant \mathcal{K} (\bs \mu, \bs \nu) + \frac{1}{k}.
\end{align*}

\noindent Compactness of $\bs \Pi(\boldsymbol{\mu},\boldsymbol{\nu})$ according to Lemma $\ref{lemmaexistenceprimal}$ implies that $(\bs \gamma^k)_{k \in \N}$ can be assumed to converge towards (say) $\bs \gamma^{\infty}$. Lower semi-continuity implies that
\begin{align*}
K(\bs \gamma^{\infty}) \leqslant \underset{k \mapsto \infty}{\liminf} K(\bs \gamma^k) \leqslant \mathcal{K} (\bs \mu, \bs \nu),
\end{align*}
\noindent and then $\bs \gamma^{\infty}$ is a minimum.
\end{proof}

\section{Duality} \label{DP}
\subsection{Presentation}

\noindent In this section, we look for a dual formulation of (KP). In order to find it, consider the following situation\footnote{This interpretation is due to L. Caffareli in scalar case, according to \cite{villani2003topics}.}: mines full of different metals ($n$ kinds) and refineries ($n$ kinds) are distributed in space. For each kind of metal corresponds a kind of refinery, for instance a kind refinery for iron, a kind of refinery for gold etc. On the one hand we want to minimize the travel cost $i.e.$ minimize the associated Kantorovich's problem, on the other hand a character suggests to supervise the travelling operation for us and propose that contract: for each ton of metal $i$ located in $x$, its price will be $\varphi_i (x)$ to extract it and for each ton of metal $j$ located in $y$ its price will be $\psi_j (y)$ to drop it off. To guarantee our interrest, its contraints will be that for all $(i,j)$ and $(x,y)$, $\varphi_i (x) + \psi_j (y) \leqslant c_{ij}(x,y)$. All these considerations suggest to give following definitions.
\begin{mydef}
Given $\mathbf{c}$ a cost matrix, $\bs \Delta (\mathbf{c})$ denotes the set of potential couples for cost $\boldsymbol{c}$ defined by
\begin{align*}
\bs \Delta (\mathbf{c}) = \left\{ \begin{array}{cc}
 &  \varphi_i \in C_b(X) \\
\begin{bmatrix}
\boldsymbol{\varphi} \\
\boldsymbol{\psi}
\end{bmatrix}= \begin{bmatrix}
\varphi_1 & \cdots & \varphi_n \\
\psi_1 & \cdots & \psi_n
\end{bmatrix}, \forall (i,j) \in \ldb 1,n \rdb^2,&  \psi_j \in C_b(Y) \\
 & \varphi_i \oplus \psi_j \leqslant c_{ij}
\end{array} \right\}
\end{align*}
and if there is no ambiguity on $\textbf{c}$, we will write $\bs \Delta$ instead of $\bs \Delta (\mathbf{c})$.
\end{mydef}

\begin{mydef}
Given $\boldsymbol{\mu} \in \mathcal{P}^n(X), \boldsymbol{\nu} \in \mathcal{P}^n(Y)$, $\mathbf{c}$ a cost matrix and $\begin{bmatrix}
\boldsymbol{\varphi} \\
\boldsymbol{\psi}
\end{bmatrix}= \begin{bmatrix}
\varphi_1 & \cdots & \varphi_n \\
\psi_1 & \cdots & \psi_n
\end{bmatrix} \in \Delta (\mathbf{c})$, $D(\boldsymbol{\varphi},\boldsymbol{\psi})$ denotes the dual cost of $\begin{bmatrix}
\boldsymbol{\varphi} \\
\boldsymbol{\psi}
\end{bmatrix}$ defined by
\begin{align}
D(\boldsymbol{\varphi},\boldsymbol{\psi}) = \sum \limits_{i=1}^{n} \int_X \varphi_{i}(x) \, \mathrm{d}\mu_i(x) + \sum \limits_{j=1}^{n} \int_Y \psi_{j}(y) \, \mathrm{d}\nu_j(y).
\end{align}
\end{mydef}

\noindent Finally, the dual transportation problem is given $\boldsymbol{\mu} \in \mathcal{P}^n(X)$, $\boldsymbol{\nu} \in \mathcal{P}^n(Y)$ and a cost matrix $\mathbf{c}$,
\begin{align*}
\tag{DP} \sup \left\{ D(\boldsymbol{\varphi},\boldsymbol{\psi}) : \begin{bmatrix}
\boldsymbol{\varphi} \\
\boldsymbol{\psi}
\end{bmatrix} \in \Delta(\boldsymbol{c}) \right\} =: \mathcal{D}(\mu,\nu) \in [- \infty, + \infty]
\end{align*}

We establish now a weak duality result.
\begin{prop} \label{weak duality}
Given $\boldsymbol{\mu} \in \mathcal{P}^n(X)$, $\boldsymbol{\nu} \in \mathcal{P}^n(Y)$, a cost matrix $\mathbf{c}, \bs \gamma \in \bs \Pi (\bs \mu,\bs \nu)$ and $\begin{bmatrix}
\boldsymbol{\varphi} \\
\boldsymbol{\psi}
\end{bmatrix} \in \bs \Delta(\bs c)$, the following inequality is satisfied,
\begin{align*}
\bs D(\boldsymbol{\varphi},\boldsymbol{\psi}) \leqslant \bs K(\bs \gamma).
\end{align*}
\end{prop}

\begin{proof}
Let $\bs \gamma \in \bs \Pi(\bs \mu,\bs \nu)$ and $\begin{bmatrix}
\boldsymbol{\varphi} \\
\boldsymbol{\psi}
\end{bmatrix} \in \bs \Delta(\bs c)$. Compute:
\begin{align*}
\bs D(\bs \varphi,\bs \psi) &= \sum \limits_{i=1}^n \int_X \varphi_i \, \mathrm{d} \mu_i + \sum \limits_{j=1}^n \int_Y \psi_j \, \mathrm{d} \nu_j \\
				   &= \sum \limits_{i=1}^n \iint_{X \times Y} \varphi_i \, \mathrm{d} \left( \sum \limits_{j=1}^n \gamma_{ij} \right) +  \sum \limits_{j=1}^n \iint_{X \times Y} \psi_j \, \mathrm{d} \left( \sum \limits_{i=1}^n \gamma_{ij} \right).
\end{align*}
The last equality coming from the fact that $\bs \gamma \in \bs \Pi(\bs \mu,\bs \nu)$. And then,
\begin{align*}
\bs D(\bs \varphi,\bs \psi) &\leqslant \sum \limits_{(i,j) \in \ldb 1,n \rdb^2} \iint_{X \times Y} c_{ij} \, \mathrm{d} \gamma_{ij} \mbox{ since $\begin{bmatrix}
\boldsymbol{\varphi} \\
\boldsymbol{\psi}
\end{bmatrix} \in \Delta(\bs c)$} \\
				   &= \bs K(\bs \gamma).
\end{align*}
That concludes the proof.
\end{proof}

\subsection{An extension of \textbf{c}-transformation}

\noindent In order to prove that (DP) is attained, at least in compact case, we propose an extension of the classical $c$-transform (see the recall below). First, we make a short digression about modulus of continuity.

\begin{mydef}
Given $(X,d)$ a metric space and $f : X \rightarrow \mathbb{R}$, a uniform modulus of continuity for $f$ according to $d$ is a function $\omega : \mathbb{R}_+ \rightarrow \mathbb{R}_+$ such that the following conditions are satisfied:

[1] $\underset{t \rightarrow 0^+}{\lim} \omega(t) = 0$

[2] $\forall (x,x') \in X^2 : |f(x) - f(x')| \leqslant \omega \left[ d(x,x') \right]$.
\end{mydef}

\begin{lemma} \label{mclemma}
If $f$ admits a uniform modulus of continuity $\omega_f$ and $g$ admits a uniform modulus of continuity $\omega_g$ then $\omega_f + \omega_g$ is a uniform modulus of continuity for $\min (f,g)$.
\end{lemma}

\begin{proof}
Let $(x,x') \in X^2$, we have
\begin{align*}
|\min (&f,g)(x) - \min (f,g)(x')| \\
&\leqslant \frac{|f(x)-f(x')|+|g(x)-g(x')|}{2} + \frac{\left| |f(x')-g(x')| - |f(x)-g(x)| \right|}{2} \\
&\leqslant \frac{\omega_f[d(x,x')]+\omega_g[d(x,x')]}{2} + \frac{|f(x')-f(x)+g(x) - g(x')|}{2} \\
&\leqslant \omega_f[d(x,x')]+\omega_g[d(x,x')].
\end{align*}
This proves the lemma.
\end{proof}
Recall that when $f$ is a function between $X$ (resp. $Y$) and $\mathbb{R} \cup \left\{- \infty \right\} $ and $c$ a cost function, we can define its $c$-transform $f^c$ (resp. $\overline{c}$-transform) by:
\begin{align*}\begin{array}{lllll}
f^c & : & Y & \rightarrow & \mathbb{R}\cup \left\{- \infty, + \infty \right\} \\
& & y & \mapsto & \inf \left\{ c(x,y) - f(x) : x \in X \right\}
\end{array} \\
\left( \mbox{ resp.} \begin{array}{lllll}
f^c & : & X & \rightarrow & \mathbb{R}\cup \left\{- \infty, + \infty \right\} \\
& & x & \mapsto & \inf \left\{ c(x,y) - f(y) : y \in Y \right\}
\end{array} \right)
\end{align*}

 We introduce a new transformation and to motivate it just remark than in our case, we have $2n$ potentials and $n^2$ inequalities in the dual formulation. A naive idea would be to first subsitute $\varphi_1$ by $\psi_1^{c_{11}}$ but there is no guarantee that our new couple of potentials $\begin{bmatrix}
\psi_1^{c_{11}} & \varphi_2 & \cdots & \varphi_n \\
\psi_1 & \cdots & \cdots &\psi_n
\end{bmatrix}$ will still be in $\bs \Delta (\bs c)$. The following definition answers this problem.

\begin{mydef}
Given $\boldsymbol{f}=(f_1,\cdots,f_n) : X \rightarrow (\mathbb{R} \cup \{- \infty \})^n$ and $\boldsymbol{c}=(c_1, \cdots, c_n) : X \times Y \rightarrow (\mathbb{R} \cup \{+ \infty \})^n$, $\boldsymbol{f}^{\boldsymbol{c}}$ (resp. $\boldsymbol{f}^{\boldsymbol{\bar{c}}}$) denotes the $\boldsymbol{c}$-transform of $f$ (resp. $\boldsymbol{\bar{c}}$-transform of $f$) defined by
\begin{align*}
\forall y \in Y : \boldsymbol{f}^{\boldsymbol{c}}(y) = \min \left( f_1^{c_1}(y),\cdots, f_n^{c_n}(y) \right) \\
\left( \mbox{resp. } \forall x \in X : \boldsymbol{f}^{\boldsymbol{\bar{c}}}(x) = \min \left( f_1^{\bar{c_1}}(x),\cdots, f_n^{\bar{c_n}}(x) \right) \right)
\end{align*}
\end{mydef}

\noindent All benefits of this transformation is contained in the next proposition.
\begin{prop} \label{proposition threetwo}
Let $\boldsymbol{f}=(f_1,\cdots,f_n) : X \rightarrow (\mathbb{R} \cup \{- \infty \})^n$ and $\boldsymbol{c}=(c_1, \cdots, c_n) : X \times Y \rightarrow (\mathbb{R} \cup \{+ \infty \})^n$, then

\textbf{[1]} Following inequalities are satisfied,
\begin{align}
\forall i \in \ldbrack 1, n \rdbrack, f_i \oplus \boldsymbol{f}^{\boldsymbol{c}} \leqslant c_i \\
\forall j \in \ldbrack 1, n \rdbrack, \boldsymbol{f}^{\boldsymbol{\bar{c}}} \oplus f_j \leqslant c_j
\end{align}

\textbf{[2]} If $h : Y \rightarrow \mathbb{R} \cup \{- \infty \} $ is such that for all $ i \in \ldbrack 1, n \rdbrack, f_i \oplus h \leqslant c_i$ then $h \leqslant \boldsymbol{f}^{\boldsymbol{c}}$. If $h : X \rightarrow \mathbb{R} \cup \{- \infty \} $ is such that for all $j \in \ldbrack 1, n \rdbrack, h \oplus g_j \leqslant c_j$ then $h \leqslant \boldsymbol{f}^{\boldsymbol{\bar{c}}}$.
\end{prop}

\begin{proof}

[1] Let $(i,j) \in \ldbrack 1, n \rdbrack^2$ and $(x,y) \in  X \times Y$. Since $f_i(x) + f_j^{c_j}(y) \leqslant c_i(x,y)$ and $\boldsymbol{f}^{\boldsymbol{c}} \leqslant f_i^{c_i}$ the first inequality is deduced and note that the second inequality can be proved following the same way.

\noindent [2] If such a function exists, we deduce from $f_i \oplus h \leqslant c_i$ that for all $(x,y)  \in X \times Y, h(y) \leqslant c_i (x,y) - f_i (x)$, then take infimum with respect to $x$ and arbitrary on $i$ concludes for the first inequality. The same proof also works for the second inequality.
\end{proof}

We will show next that this process is a natural way to improve the dual cost while staying in the constraint $\Delta(\bs c)$, at least in compact case and continuous costs. Moreover, it provides a common uniform modulus of continuity for all the potentials.

\begin{lemma} \label{acmc}
Let $X,Y$ two compact metric spaces, \textbf{c} a continuous cost matrix and $(
\boldsymbol{\varphi},
\boldsymbol{\psi} ) \in \bs \Delta(\boldsymbol{c})$. It exists $(
\underline{\boldsymbol{\varphi}},
\underline{\boldsymbol{\psi}} ) \in \bs \Delta(\boldsymbol{c})$ such that

\textbf{[1]} $D(\boldsymbol{\varphi},\boldsymbol{\psi}) \leqslant D(\underline{\boldsymbol{\varphi}},\underline{\boldsymbol{\psi}})$.

\textbf{[2]} $\underline{\varphi_1},\cdots,\underline{\varphi_n},\underline{\psi_1},\cdots,\underline{\psi_{n-1}}$ and $\underline{\psi_n}$ admit a common uniform modulus of continuity which depends only on $\textbf{c}$.
\end{lemma}

\begin{proof}
First, make the following substitutions:
\begin{align*}
\forall j \in \ldbrack 1,n \rdbrack : \psi_j \leftarrow \varphi^{(c_{1j},\cdots,c_{nj})} := \underline{\psi_j},
\end{align*}
then, thanks to Proposition \ref{proposition threetwo}, $(\boldsymbol{\varphi},\underline{\boldsymbol{\psi}}) \in \bs \Delta(\bs c)$ and $D(\boldsymbol{\varphi},\boldsymbol{\psi}) \leqslant D(\boldsymbol{\varphi},\underline{\boldsymbol{\psi}})$. Denoting $\omega_{c_{ij}}$ a uniform modulus of continuity of $c_{ij}$ for $(i,j) \in \ldbrack 1,n \rdbrack^2$, $\omega_{c_{ij}}$ is also a uniform modulus of continuity of $\varphi_i^{c_{ij}}$ according to \cite{santambrogio2015optimal} (Box. 1.8). Thanks to Lemma \ref{mclemma}, we conclude that $\omega_{\underline{\psi_j}} = \omega_{c_{1j}} + \cdots + \omega_{c_{nj}}$ is a uniform modulus of continuity of $\underline{\psi_j}$. Then, make the following substitutions:
\begin{align*}
\forall i \in \ldb 1,n \rdb : \varphi_i \leftarrow \underline{\psi}^{\overline{(c_{i1},\cdots,c_{in})}} := \underline{\varphi_i}
\end{align*}
and of course the new couple of potentials is still in $\bs \Delta(c)$ and the dual cost is increased. To conclude, we just have to check that $\sum_{1 \leqslant i,j \leqslant n} \omega_{c_{ij}}$ is a common uniform modulus of continuity for $(\underline{\boldsymbol{\varphi}},\underline{\boldsymbol{\psi}})$, which is clear.
\end{proof}

\begin{ex} \label{big example}
Coming back to the example \ref{basic example}, let us compute this new $\bs c$-transform to reduce the problem. Fix $\kappa$ to be strictly non-negative and assume that $X=Y$ and $\bs c$ is symetric (then, $\bs c$-transform is equivalent to $\bs{\overline{c}}$-transform). Constraints of (DP) are given by the following system:
$$
\left\{
\begin{split}
\varphi_1(x) + \psi_1(y) &\leqslant c(x,y) \\
\varphi_1(x) + \psi_2(y) &\leqslant c(x,y) + \kappa \\
\varphi_2(x) + \psi_1(y) &\leqslant c(x,y) + \kappa \\
\varphi_2(x) + \psi_2(y) &\leqslant c(x,y)
\end{split}
\right.
$$
\textbf{First step:} it is easy to check that:
\begin{align*}
(f_1,f_2)^{c,c + \kappa} = \left[ f_1 \wedge (f_2 - \kappa)  \right]^c \\
(f_1,f_2)^{c+\kappa,c} = \left[ (f_1 - \kappa) \wedge f_2 \right]^c,
\end{align*}
then make the following substitutions:
\begin{align*}
\psi_1 \leftarrow (\varphi_1,\varphi_2)^{c,c+\kappa} = \left[ \varphi_1 \wedge (\varphi_2 - \kappa)  \right]^c =: \tilde{\psi_1} \\
\psi_2 \leftarrow (\varphi_1,\varphi_2)^{c+\kappa,c} = \left[ (\varphi_1 - \kappa) \wedge \varphi_2  \right]^c =: \tilde{\psi_2}
\end{align*}
\noindent \textbf{Second step:} following the proof below, we make the following substitutions:
\begin{align*}
\varphi_1\leftarrow (\tilde{\psi_1},\tilde{\psi_2})^{c,c+\kappa} &= \left[ \tilde{\psi_1} \wedge (\tilde{\psi_2} - \kappa)  \right]^c \\
&= \tilde{\psi_1}^c \vee (\tilde{\psi_2} - \kappa)^c  \mbox{  since $(\underset{\alpha}{\sup} f_{\alpha} )^c = \underset{\alpha}{\inf} f_{\alpha}^c$} \\
&= \left[ \varphi_1 \wedge (\varphi_2 - \kappa)  \right]^{cc} \vee (\left[ (\varphi_1 - \kappa) \wedge \varphi_2  \right] - \kappa)^{cc} \\
&= \left[ \varphi_1 \wedge (\varphi_2 - \kappa)  \right]^{cc} \vee \left[ (\varphi_1 - 2 \kappa) \wedge \left( \varphi_2 - \kappa \right) \right]^{cc} \\
&= \left[ \varphi_1 \wedge (\varphi_2 - \kappa)  \right]^{cc} \mbox{ since if $f \leqslant g$ then $g^c \leqslant f^c$} \\
&= \tilde{\psi_1}^c \\
\varphi_2 \leftarrow (\tilde{\psi_1},\tilde{\psi_2})^{c +\kappa,c} &= \tilde{\psi_2}^c \mbox{ for the same reasons.}
\end{align*}

\noindent When $c = d$ is a distance, according to \cite{santambrogio2015optimal} (Proposition 3.1):
$$
\left\{
\begin{split}
\tilde{\psi_1}(y) - \tilde{\psi_1}(x) &\leqslant d(x,y) \\
\tilde{\psi_2}(y) - \tilde{\psi_1}(x) &\leqslant d(x,y) + \kappa \\
\tilde{\psi_1}(y) - \tilde{\psi_2}(x) &\leqslant d(x,y) + \kappa \\
\tilde{\psi_2}(y) - \tilde{\psi_2}(x) &\leqslant d(x,y),
\end{split}
\right.
$$

\noindent which is equivalent to the following system, thanks to the symmetry of $d$: 
$$
\left\{
\begin{split}
| \tilde{\psi_1}(x) - \tilde{\psi_1}(y) | &\leqslant d(x,y)\\
| \tilde{\psi_1}(x) - \tilde{\psi_2}(y) | &\leqslant d(x,y) + \kappa\\
| \tilde{\psi_2}(x) - \tilde{\psi_2}(y) | &\leqslant d(x,y)
\end{split}
\right.
$$

\noindent $i.e.$ $(\tilde{\psi_1},\tilde{\psi_2})$ are solution to the system below if and only if they are $1$-Lipschitz w.r.t. to $d$ and satisfy $\|\tilde{\psi_1} - \tilde{\psi_2} \|_{\infty} \leqslant \kappa$.
\end{ex}

\subsection{Existence of a maximizer}

\begin{theorem} \label{existence th in D}
Given $X$ et $Y$ two compact metric spaces, $\boldsymbol{\mu} \in \mathcal{P}^n(X)$, $\boldsymbol{\nu} \in \mathcal{P}^n(Y)$ and \textbf{c} a continuous cost matrix, there exists $(
\boldsymbol{\varphi},
\boldsymbol{\psi} ) \in \bs \Delta(\boldsymbol{c})$ such as $\mathcal{D}(\boldsymbol{\mu},\boldsymbol{\nu}) = D(\boldsymbol{\varphi},\boldsymbol{\psi})$.
\end{theorem}

\begin{proof}
The constraint set is non-empty since $\bs c$ is bounded by below (continuous on compact). Let:
\begin{align*}
\begin{bmatrix}
\bs{\varphi^k} \\
\boldsymbol{\psi^k}
\end{bmatrix}_{k \in \mathbb{N}}= \begin{bmatrix}
\varphi_1^k & \cdots & \varphi_n^k \\
\psi_1^k & \cdots & \psi_n^k
\end{bmatrix}_{k \in \mathbb{N}}
\end{align*}
be a maximizing sequence for (DP). According to Lemma \ref{acmc}, we may assume that our $2n$ sequences share a common uniform modulus of continuity. We now prove that the sequence is uniformly bounded with respect to $n$. Indeed, setting for all $k \in \mathbb{N}$:
\begin{align*}
m_k := \min \left[ \underset{x \in X}{\inf}\varphi_1^k(x),\cdots,\underset{x \in X}{\inf}\varphi_n^k(x) \right],
\end{align*}
and since $m_k$ is finite, we can substitute:
\begin{align*}
\forall i \in \ldbrack 1, \cdots,n \rdbrack :  \varphi_i^k \leftarrow \varphi_i^k - m_k \mbox{ still written $\varphi_i^k$} \\
\forall j \in \ldbrack 1, \cdots,n \rdbrack : \psi_j^k \leftarrow \psi_j^k + m_k \mbox{ still written $\psi_j^k$},
\end{align*}
and these new potentials are still admissible, have the same dual cost and for all $i \in \ldbrack 1,n \rdbrack, k \in \mathbb{N}, \varphi_i^k \geqslant 0$. Therefore we have:
\begin{align*}
\forall i \in \ldb 1,n \rdb, k \in \mathbb{N} : \varphi_i^k \leqslant \omega \left[ diam(X) \right],
\end{align*}
which concludes the case of ${\boldsymbol{\varphi}}$. Next, let us make new following substitutions:
\begin{align*}
\forall j \in \ldbrack 1,n \rdbrack : \psi_j \leftarrow \varphi^{\overline{(c_{1j},\cdots,c_{nj})}} \mbox{ still written $\psi_j$}.
\end{align*}
We have for all $y \in Y, j \in \ldbrack 1,n \rdbrack$ and $k \in \mathbb{N}$,
\begin{align*}
\min &\left( c_{1j},\cdots c_{nj} \right) - \omega \left[ diam(X) \right] \\
&\leqslant \min \left[ \underset{x \in X}{\inf} c_{1j}(x,y) - {\varphi_1^k}(x), \cdots, \underset{x \in X}{\inf} c_{nj}(x,y) - \varphi_n^k(x)  \right] := {\psi_j^k}(y) \\
\mbox{and   } {\psi_j^k}(y) &:= \min \left[ \underset{x \in X}{\inf} c_{1j}(x,y) - {\varphi_1^k}(x), \cdots, \underset{x \in X}{\inf} c_{nj}(x,y) - \varphi_n^k(x)  \right] \\
&\leqslant \max \left( c_{1j},\cdots, c_{nj} \right),
\end{align*}
which leads to the conclusion on $\bs \psi$. Finally, the Ascoli-Arzelà theorem applied to each sequence provides the existence of a continuous couple
\begin{align*}
\begin{bmatrix}
\boldsymbol{\varphi^{\infty}} \\
\boldsymbol{\psi^{\infty}}
\end{bmatrix} = \begin{bmatrix}
\varphi_1^{\infty} & \cdots & \varphi_n^{\infty} \\
\psi_1^{\infty} & \cdots & \psi_n^{\infty}\end{bmatrix}
\end{align*}
which belong to $\Delta(\boldsymbol{c})$ thanks to pointwise convergence and $D(\boldsymbol{\varphi^{\infty}},\boldsymbol{\psi^{\infty}}) = \mathcal{D}(\boldsymbol{\mu},\boldsymbol{\nu})$ thanks to uniform convergence on finite measure sets.
\end{proof}

\subsection{Strong duality} \label{SDT}

\noindent We establish a strong duality result. The proof follows the one of strong duality theorem for scalar optimal transportation proposed by C. Jimenez (see \cite{santambrogio2015optimal}).

\begin{mydef}
Given $\bs \mu \in \mathcal{P}^n(X)$, $\bs \nu \in \mathcal{P}^n(Y)$ and $\bs c$ a cost matrix,  we denote by $H$ the value function of the perturbated dual problem, $i.e.$
\begin{align*}
\forall \bs \varepsilon \in C(X \times Y,\mathbb{R}^{n \times n}), H(\bs\varepsilon)= \sup \left\{ D(\bs\varphi,\bs\psi) : \begin{bmatrix}
\boldsymbol{\varphi} \\
\boldsymbol{\psi}
\end{bmatrix} \in \Delta(\bs c -\bs\varepsilon) \right\}
\end{align*}
\end{mydef}

\begin{lemma} Let $X$ and $Y$ two metric compact spaces. $H$ satisfy the following properties:

\textbf{[1]} $H$ is concave.

\textbf{[2]} Suppose that \textbf{c} is continuous, then $H$ is u.s.c. with respect to the uniform norm.
\end{lemma}

\begin{proof}

[1] Let $t \in [0,1]$, $\bs{\varepsilon^0} \in C(X \times Y, \mathbb{R}^{n \times n})$ (resp. $\bs{\varepsilon^1} \in C(X \times Y, \mathbb{R}^{n \times n})$) and let $(\bs{\varphi^0},\bs{\psi^0})$ (resp. $(\bs{\varphi^1},\bs{\psi^1}))$ be optimal in (DP) associated to $\bs c - \bs{\varepsilon^0}$ (resp. to $\bs c-\bs{\varepsilon^1}$). Note that they exist thanks to the existence result below. Define $\bs{\varepsilon_t} = (1-t)\bs{\varepsilon^0} + t \bs{\varepsilon^1},\bs{\varphi_t} = (1-t)\bs{\varphi^0} + t \bs{\varphi^1},\bs{\psi_t} = (1-t)\bs{\psi^0} + t \bs{\psi^1}$. Therefore $(\bs{\varphi_t},\bs{\psi_t})$ is admissible for the dual problem associated to $\bs c-\bs{\varepsilon_t}$ and then by definition of $H$ we have
\begin{align*}
H(\bs{\varepsilon_t}) \geqslant D(\bs{\varphi_t},\bs{\psi_t}) =(1-t) H(\bs{\varepsilon^0}) + t H(\bs{\varepsilon^1})
\end{align*}
And the conclusion follows.

\noindent [2] Let $(\bs{\varepsilon^k})_{k \in \mathbb{N}} \in C(X \times Y, \mathbb{R}^{n \times n})^{\mathbb{N}}$ and $\bs{\varepsilon^{\infty}}  \in C(X \times Y, \mathbb{R}^{n \times n})$ such that for all $(i,j) \in \ldbrack 1, n \rdbrack^2, \varepsilon^k_{ij} \overset{\|\cdot \|_{\infty}}{\longrightarrow} \varepsilon^{\infty}_{ij}$. Let $(\bs{\varepsilon^k})_{k \in \mathbb{N}}$ a subsequence $(\bs{\varepsilon^{k_l}})_{l \in \mathbb{N}}$ satisfying for all $(i,j) \in \ldbrack 1, n \rdbrack^2, \underset{k}{\limsup} \, H(\varepsilon^k_{ij}) = \underset{l}{\lim} \, H(\varepsilon^{k_l}_{ij})$. Ascoli-Arzelà theorem ensures that for all $(i,j) \in \ldbrack 1, n \rdbrack^2, (\varepsilon^{k_l}_{ij})_{l \in \mathbb{N}}$ are equicontinuous and uniformly bounded with respect to $l$ therefore chose the corresponding optimal potentials $(\bs{\varphi^{k_l}},\bs{\psi^{k_l}})$ equicontinuous and uniformly bounded in $l$, thanks to the Ascoli-Arzelà theorem again, there is a uniform convergence towards (say) $(\bs{\varphi^{\infty}},\bs{\psi^{\infty}})$ up to an extraction, therefore thanks to pointwise convergence there is for all $(i,j) \in \ldbrack 1, n \rdbrack^2, \varphi^{\infty}_i \oplus \psi^{\infty}_j \leqslant c_{ij} - \varepsilon_{ij}$ and so:
\begin{align*}
H(\bs{\varepsilon}) \geqslant D(\bs{\varphi^{\infty}},\bs{\psi^{\infty}}) = \underset{l}{\lim} \, H(\bs{\varepsilon^{k_l}}) = \underset{k}{\limsup} \, H(\bs{\varepsilon^{k}}).
\end{align*}
This concludes the proof.
\end{proof}

\noindent Finally, the strong duality theorem follows.
\begin{theorem} \label{SDT}
Suppose that $X$ and $Y$ are both metric compact spaces and that $\bs c$ is continuous, then for all $(\bs\mu,\bs\nu) \in \mathcal{P}^n(X) \times \mathcal{P}^n(Y), \mathcal{K}(\bs\mu,\bs\nu) = \mathcal{D}(\bs\mu,\bs\nu)$.
\end{theorem}

\begin{proof}
Let $(\bs\mu,\bs\nu) \in \mathcal{P}^n(X) \times \mathcal{P}^n(Y)$, since $\left( - H \right)$ is convex and l.s.c. and according to the Fenchel-Moreau theorem, we have:
\begin{align*}
\mathcal{D}(\bs\mu,\bs\nu) &= H(\bs 0) \\
						   &= - [ - H(\bs 0)] \\
						   &= - [- H]^{**}(\bs 0) \\
						   &= - \underset{\bs \gamma \in \mathcal{M}^{n \times n}(X \times Y)}{\sup} <\bs 0, \bs \gamma > - [- H]^*(\bs \gamma) \\
						   &= \underset{\bs \gamma \in \mathcal{M}^{n \times n}(X \times Y)}{\inf} [-H]^*(\bs \gamma).
\end{align*}

\noindent Next, we compute for all $\bs \gamma \in \mathcal{M}^{n \times n}(X \times Y)$,
\begin{align*}
[-H&]^*(\bs \gamma) = \underset{\bs \varepsilon}{\sup} \left( \sum \limits_{1 \leqslant i,j \leqslant n} \iint_{X \times Y} \varepsilon_{ij} \, \mathrm{d}\gamma_{ij} + \underset{\bs \varphi \bs \oplus \bs \psi \leqslant \bs c - \bs \varepsilon}{\sup}  \sum \limits_{i = 1}^{n} \int_{X} \varphi_i \, \mathrm{d}\mu_{i} + \sum \limits_{j = 1}^{n} \int_{Y} \psi_j \, \mathrm{d}\nu_{j} \right) \\
				&= \underset{\bs \varepsilon}{\sup} \left( \underset{\bs \varphi \bs \oplus \bs \psi \leqslant \bs c - \bs \varepsilon}{\sup} \left[ \sum \limits_{1 \leqslant i,j \leqslant n} \iint_{X \times Y} \varepsilon_{ij} \, \mathrm{d}\gamma_{ij} + \sum \limits_{i = 1}^{n} \int_{X} \varphi_i \, \mathrm{d}\mu_{i} + \sum \limits_{j = 1}^{n} \int_{Y} \psi_j \, \mathrm{d}\nu_{j} \right] \right).
\end{align*}

\noindent If it exists $(i_0,j_0)$ such as $\gamma_{i_0 j_0} \notin \mathcal{M}_+(X \times Y)$ and $\varepsilon_{i_0 j_0}^0$ such as $\iint_{X \times Y} \varepsilon_{i_0 j_0}^0 \mathrm{d} \gamma_{i_0 j_0} > 0 $, then take $(\varepsilon_{i_0 j_0}^k)_{k \in \mathbb{N^*}}$ such as $\varepsilon_{i_0 j_0}^k = c_{i_0 j_0} + k \varepsilon_{i_0 j_0}^0$ for $k \in \mathbb{N}^*$ and take $\varphi_{i_0} = 0$ and $\psi_{j_0} = 0$. Then putting all the other potentials equals at the value 0 and find $(\varepsilon_{ij})$ such that all the contraints are still satisfied ($\mathbf{c}$ is bounded), we get $[-H]^*(\bs \gamma) = + \infty$ if $\bs \gamma \notin \mathcal{M}_+^{n \times n}(X \times Y)$.
Now, suppose that $\bs \gamma \in \mathcal{M}_+^{n \times n}(X \times Y)$, when $(\bs \varphi,\bs \psi)$ are fixed, we are interested in taking the largest $\varepsilon_{ij}$ possible for every $(i,j) \in \ldbrack 1,n \rdbrack^2$, that is $\varepsilon_{ij}= c_{ij} - \varphi_i - \psi_j$ and we get
\begin{multline*}
[-H]^*(\bs \gamma) = \underset{\bs \varphi \bs \oplus \bs \psi \leqslant \bs C - \bs \varepsilon}{\sup} \sum \limits_{1 \leqslant i,j \leqslant n} \iint_{X \times Y} c_{ij} - \varphi_i - \psi_j \, \mathrm{d}\gamma_{ij} \\
+ \sum \limits_{i = 1}^{n} \int_{X} \varphi_i \, \mathrm{d}\mu_{i} + \sum \limits_{j = 1}^{n} \int_{Y} \psi_j \, \mathrm{d}\nu_{j} \\
					= \underset{(\bs \varphi, \bs \psi)}{\sup} K(\bs \gamma) + \left( \sum \limits_{i = 1}^{n} \int_{X} \varphi_i \, \mathrm{d}\mu_{i} - \sum \limits_{1 \leqslant i,j \leqslant n} \iint_{X \times Y} \varphi_{i} \, \mathrm{d}\gamma_{ij} \right) \\ 
					+ \left( \sum \limits_{j = 1}^{n} \int_{Y} \psi_j \, \mathrm{d} \nu_{j} - \sum \limits_{1 \leqslant i,j \leqslant n} \iint_{X \times Y} \psi_{j} \, \mathrm{d}\gamma_{ij} \right) \\
					= \underset{(\bs \varphi, \bs \psi)}{\sup} K(\bs \gamma) +  \sum  \limits_{i = 1}^{n} \left( \int_{X} \varphi_i \, \mathrm{d}\mu_{i} - \iint_{X \times Y} \varphi_{i} \, \mathrm{d}  \sum \limits_{j=1}^{n} \gamma_{ij} \right) \\
					+  \sum \limits_{j = 1}^{n} \left( \int_{Y} \psi_j \, \mathrm{d} \nu_{j} - \iint_{X \times Y} \psi_{j} \, \mathrm{d} \sum \limits_{i=1}^{n} \gamma_{ij} \right) \\
= \iota_{\Pi(\bs \mu,\bs \nu)}(\bs \gamma) \mbox{ according to \cite{santambrogio2015optimal}, Lemma 1.45}.
\end{multline*}
This ends the proof.
\end{proof}

\section{Optimality conditions} \label{OC}
In this subsection, $X$ and $Y$ are two metric compact spaces. As a direct consequence of Theorem 3.2, we deduce optimality contraints linking (KP) and (DP).
\begin{prop}
Given $\boldsymbol{\gamma} \in \bs \Pi(\bs \mu, \bs \nu)$ and $(
\bs{\varphi},
\boldsymbol{\psi}) \in \bs \Delta(\bs c)$, the following assertions are equivalent:

\textbf{[1]} $\boldsymbol{\gamma}$ is optimal in (KP) and $(\bs{\varphi},\boldsymbol{\psi}
)$ is optimal in (DP).

\textbf{[2]} $\forall (i,j), \varphi_i \oplus \psi_j = c_{ij}$ $\gamma_{ij}$-a.e.
\end{prop}

\begin{proof}
If \textbf{[1]} is satisfied, according to Theorem \ref{SDT}, $
K(\boldsymbol{\gamma}) = D(\boldsymbol{\varphi},\boldsymbol{\psi})$. We then compute $D(\boldsymbol{\varphi},\boldsymbol{\psi})$ as a function of $\bs \gamma$.
\begin{align*}
D(\boldsymbol{\varphi},\boldsymbol{\psi}) :&= \sum \limits_{i =1}^{n} \int_X \varphi_{i}(x) \, \mathrm{d}\mu_i(x) + \sum \limits_{j =1}^{n} \int_Y \psi_{j}(x) \, \mathrm{d}\nu_j(x) \\
&= \sum \limits_{1 \leqslant i,j \leqslant n} \iint_{X \times Y} \varphi_{i}(x) \, \mathrm{d}\gamma_{ij}(x,y) + \sum \limits_{1 \leqslant i,j \leqslant n} \iint_{X \times Y} \psi_{j}(x) \, \mathrm{d}\gamma_{ij}(x,y)  \\
&= \sum \limits_{1 \leqslant i,j \leqslant n} \iint_{X \times Y} \left[ \varphi_i(x) + \psi_j(y) \right] \, \mathrm{d} \gamma_{ij}(x,y).
\end{align*}
\noindent Comparing the latter expression with $K( \bs \gamma)$ gives
\begin{align*}
0 &= K(\boldsymbol{\gamma}) - D(\boldsymbol{\varphi},\boldsymbol{\psi}) \\
	&= \sum_{ij} \iint_{X \times Y} \left( c_{ij}(x,y) - \left[ \varphi_i(x) + \psi_j(y) \right] \right) \, \mathrm{d} \gamma_{ij}(x,y).
\end{align*}
\noindent The conclusion follows from the fact that $(
\bs{\varphi},
\boldsymbol{\psi}) \in \bs \Delta(\bs c)$.

Conversely, if \textbf{[2]} is satisfied, it is clear that $K(\boldsymbol{\gamma}) = D(\boldsymbol{\varphi},\boldsymbol{\psi})$ which implies that both $ \bs \gamma$ and $(\bs \varphi,\bs \psi)$ are optimal according to Proposition \ref{weak duality}.
\end{proof}

\noindent The result above is not surprising since any given $\bs \gamma \in \bs \Pi(\bs \mu,\bs \nu)$ induces $n^2$ $scalar$ optimal transportation problems between each marginals (say) $\pi_1 \# \gamma_{ij} := f_{ij} \mathrm{d} \mu_i$ and $\pi_2 \# \gamma_{ij} := g_{ij} \mathrm{d} \nu_j$
\begin{align}
\inf \left\{ \iint_{X \times Y} c_{ij}(x,y) \, \mathrm{d} \gamma_{ij}(x,y) : \gamma_{ij} \in \Pi(f_{ij} \mathrm{d} \mu_i,g_{ij} \mathrm{d} \nu_j) \right\} \tag{$KP_{ij}$},
\end{align}

\noindent and looking at contraints in vectorial Kantorovitch's problem, it is easy to see that $\bs \gamma$ has to be optimal in every subproblems ($KP_{ij}$) to be optimal between $\bs \mu$ and $\bs \nu$ (if not, take a better one and compare the total cost, which is nothing less than another proof of the result above).

\section{Induced metrics} \label{Topology}

In this section, we take $X=Y$ a Polish space. We investigate how to extend the well-known Wasserstein distance and answer the question ''does the problem (KP) define a distance on the space $\mathcal{P}^n(X)$?''.

Let $(d_{ij})_{(i,j) \in \ldb 1,n \rdb^2}$ be $n^2$ finite, symmetric and non negative functions on $X \times X$ satisfying the triangle inequality (we do not assume that they are distances). Then let $p \in [1, \infty)$, $x_0 \in X$ and define
\begin{align*}
\mathcal{P}^n_p(X) = \left\{ \bs m \in \mathcal{P}^n(X), \forall (i,j) \in \ldbrack 1,n \rdbrack^2, \int_X d_{ij}^p (x_0,x) + d_{ji}^p (x_0,x) \, \mathrm{d} m_i (x) < \infty \right\}
\end{align*}

\noindent and notice that as in scalar case, this set does not depend on $x_0$.
\begin{mydef}
Given $\bs \mu, \bs \nu \in \mathcal{P}^n_p (X)$, the $p$-transportation distance between $\bs \mu$ and $\bs \nu$ is defined by
\begin{align*}
W_p(\bs \mu, \bs \nu ) &= \left(  \inf \left\{ \sum \limits_{ij} \iint d_{ij} (x,y)^p \, \mathrm{d} \gamma_{ij} (x,y), \bs \gamma \in \bs \Pi (\bs \mu, \bs \nu ) \right\} \right)^{\frac{1}{p}} \\
&:= \left( \mathcal{K}(\bs \mu, \bs \nu) \right)^{\frac{1}{p}}.
\end{align*}
\end{mydef}
The symmetry of $W_p$ is clear provided the costs are symmetric themselves. However, the fact that $W_p(\bs \mu, \bs \nu)=0$ implies that $\bs \mu = \bs \nu$ is never satisfied if all costs are (power of) distances. In place of it, if $W_p(\bs \mu, \bs \nu)=0$ then $\sum_{i=1}^{n} \mu_i = \sum_{j=1}^{n} \nu_j$. In other words, $\bs{W_p}$ is pseudodistance in that case. To prevent that, we add new hypothesis on $(d_{ij})$ described in the next proposition.

\begin{prop}
Let $(d_{ij})_{(i,j) \in \ldbrack 1, n \rdbrack^2}$ be $n^2$ symetric finite non negative functions on $X \times X$ satisfyong the triangle inequality. Assume moreover that for all $(i,j) \in \ldbrack 1, n \rdbrack^2, i \neq j, d_{ij}$ is strictly non negative and $d_{ii}$ is a distance. Then for all $\bs \mu, \bs \nu \in \mathcal{P}^n_p (X)$, if $W_p(\bs \mu, \bs \nu )=0$ then $\bs \mu = \bs \nu$. 
\end{prop}

\begin{proof}
Let $\bs \mu, \bs \nu \in \mathcal{P}^n_p(X)$ be such as $W_p(\bs \mu, \bs \nu) = 0$ and let $\bs{\gamma^*}$ be optimal in (KP), then
\begin{align*}
0 = \sum \limits_{k=1}^{n} \iint_{X^2} d_{kk} (x,y)^p \, \mathrm{d} \gamma_{kk}^* (x,y) + \sum \limits_{i \neq j} \iint_{X^2} d_{ij} (x,y)^p \, \mathrm{d} \gamma_{ij}^* (x,y)
\end{align*}

\noindent According to the strict positivity of non diagonal distances, for all $i \neq j$, $\gamma_{ij}^* = 0$ and then for all $k \in \ldb 1,n \rdb, \gamma_{kk}^*$ is a transport plan between $\mu_k$ and $\nu_k$. The proof of Theorem 7.3. in \cite{villani2003topics} concludes.
\end{proof}
\noindent However, without any other constraints on $(d_{ij})$, the following example shows that the triangle inequality fails.
\begin{ex} \label{contrexemple distance}
Let $X = \mathbb{R}, n=2$ and set:
\begin{align*}
\mu =\begin{bmatrix}
\delta_{0} \\
0
\end{bmatrix}, \bs \nu  = \begin{bmatrix}
0 \\
\delta_{1}
\end{bmatrix},\bs \lambda = \begin{bmatrix}
\delta_{2} \\
0
\end{bmatrix}
\end{align*}
Then set $p=1, d_{11} = | \cdot |$ and $d_{12} = d_{21} := d_{\varepsilon}$ the $\varepsilon$-discrete distance on $\mathbb{R}$ (with $\varepsilon \in \mathbb{R}_+^*$) defined by $d_{\varepsilon} (x,y) = \varepsilon$ if $x=y$ and $0$ otherwise and an arbitrary distance for $d_{22}$. Clearly,
\begin{align}
W_1(\bs \mu, \bs \lambda) = 2, W_1(\bs \mu, \bs \nu) = W_1(\bs \nu, \bs \lambda) = \varepsilon
\end{align}
\noindent And these three numbers do not satisfy to triangle inequality as soon as $\varepsilon$ is smaller enough.
\end{ex}

The main problem in the example above is the lack of comparison between all $(d_{ij})$. To give a everyday-life example, it could be more expansive to travel between Paris and Berlin using plane than to first travel between Paris and Amsterdam using car and then going to Berlin from Amsterdam using train. To avoid this phenomenon above, we add new constraints on $(d_{ij})$:
\begin{align*}
\tag{MTI} \forall (i,j,k) \in \ldbrack 1,n \rdbrack^3, \forall (x,y,z) \in X^3, d_{ik}(x,z) \leqslant d_{ij} (x,y) + d_{jk} (y,z)
\end{align*}
\noindent and from now on we assume that these contraints are satisfied.
\begin{rem} Note that $(MTI)$ (for Mixed Triangle Inequalities) contain the fact that all costs satisfy triangle inequality (take $i=j=k$) and if one of theses inequalities is false for some $(x_0,y_0,z_0)$ then one can exhibit a counterexample to fail the triangle inequality on $W_p$ similar to the (counter)Example \ref{contrexemple distance} above.
\end{rem}

\begin{ex}
An easy way to construct objects that satisfy (MTI) is (and then, we do not work on empty set) given a distance $d$ on $X$ and a non negative scalar $t$ (for transformation), $d_{ii} = d$ for all $i$ and $d_{ij} = d + t$ for all $(i,j)$ with $i \neq j$.
\end{ex}

\begin{prop}
Let $p \in [1,\infty )$ and $(d_{ij})_{(i,j) \in \ldbrack 1,n \rdbrack^2}$ be such that (MTI) are satisfied. Then $W_p$ satisfies the triangle inequality.
\end{prop}

\begin{proof}
Let $\bs{\gamma^*} = (\gamma_{ij}^*)$ (resp. $\bs{\tilde{\gamma}^*} = (\tilde{\gamma}_{jk}^*)$) be optimal\footnote{they exist according to \ref{Existence th in K}, even if it is not necessary here: passing to supremum bound $a posteriori$ otherwise.} between $\bs \mu$ and $\bs \nu$ (resp. $\bs \nu$ and $\bs \lambda$). Let $j \in \ldbrack 1,n \rdbrack$ and define for all $i,k \in \ldbrack 1,n \rdbrack$ the marginals $\nu^{i,\leftarrow}_j := \pi_2 \# \gamma_{ij}^*$ and $\nu^{k,\rightarrow}_j := \pi_1 \# \tilde{\gamma}_{jk}^*$. These marginals are all submeasures of $\nu_j$ and then, according to Radon-Nikodym theorem, we denote by $f^{i,\leftarrow}_j$ (resp. $f^{k,\rightarrow}_j$) the density of $\nu^{i,\leftarrow}_j$ (resp. $\nu^{k,\rightarrow}_j$) w.r.t. $\nu_j$. Finally, define for each $i,j,k \in \ldbrack 1,n \rdbrack$ the following transference plans
\begin{align}
&\gamma_{ijk}^* \mbox{ is defined as the measure with density $(x,y) \rightarrow f^{k,\rightarrow}_j (y)$ w.r.t. $\gamma_{ij}^*$}, \label{11} \\
&\tilde{\gamma}_{ijk}^* \mbox{ is defined as the measure with density $(y,z) \rightarrow  f^{i,\leftarrow}_j (y)$ w.r.t. $\tilde{\gamma}_{jk}^*$}, \label{12}
\end{align}
\noindent these definitions imply that
\begin{align}
&\forall (i,j)  \in \ldbrack 1,n \rdbrack^2, \gamma_{ij}^* =\sum \limits_{k = 1}^{n} \gamma_{ijk}^*, \label{13} \\
&\forall (j,k) \in \ldbrack 1,n \rdbrack^2, \tilde{\gamma}_{jk}^* =\sum \limits_{i = 1}^{n} \tilde{\gamma}_{ijk}^*, \label{14}\\
&\forall (i,j,k) \in \ldbrack 1,n \rdbrack^3, \pi_2 \# \gamma_{ijk}^* = \pi_1 \# \tilde{\gamma}_{ijk}^*. \label{Magic equality}
\end{align}

\noindent To obtain the last equality, fix $B$ a measurable subset of $Y$ and compute
\begin{align*}
\pi_2 \# \gamma_{ijk}^* (B) &= \gamma_{ijk}^* (X \times B) \\
&= \iint_{X \times B} f^{k,\rightarrow}_j (y) \, \mathrm{d} \gamma_{ij}^* (x,y) \mbox{ by (\ref{11})}\\
&= \int_B  f^{k,\rightarrow}_j (y)  \, \mathrm{d} \nu^{i,\leftarrow}_j (y) \mbox{ by definition of $\nu^{i,\leftarrow}_j$} \\
&= \int_B  f^{k,\rightarrow}_j (y) f^{i,\leftarrow}_j (y) \, \mathrm{d} \nu_j  (y) \mbox{ by definition of $f^{i,\leftarrow}_j$} \\
&= \int_B  f^{i,\leftarrow}_j (y) \, \mathrm{d} \nu^{k,\rightarrow}_j  (y) \mbox{ by definition of $f^{k,\rightarrow}_j$} \\
&= \iint_{B \times Z} f^{i,\leftarrow}_j (y) \, \mathrm{d} \tilde{\gamma}_{ij}^* (y,z) \mbox{ by definition of $\nu^{k,\rightarrow}_j$} \\
&= \pi_1 \# \tilde{\gamma}_{ijk}^* (B) \mbox{ by (\ref{12})}.
\end{align*}

\noindent Then, equalities (\ref{Magic equality}) allow us to apply the Gluing Lemma (see \cite{villani2003topics}, Lemma 7.6) and guarantee the existence of $\Pi_{ijk}$ a measure on $X \times Y \times Z$ such that
\begin{align*}
\pi_{1,2} \# \Pi_{ijk} = \gamma_{ijk}^* \mbox{ and } \pi_{2,3} \# \Pi_{ijk} = \tilde{\gamma}_{ijk}. \label{GL}
\end{align*}

\noindent We next define for all $(i,k) \in \ldbrack 1,n \rdbrack^2, \Pi_{ik} = \sum \limits_{j=1}^{n} \Pi_{ijk}$ and compute
\begin{align*}
\pi_1 \# \sum \limits_{k=1}^{n} \Pi_{ik} &= \sum \limits_{k=1}^{n} \sum \limits_{j=1}^{n} \pi_1 \# \Pi_{ijk} \mbox{ by definition}\\
&= \sum \limits_{k=1}^{n} \sum \limits_{j=1}^{n} \pi_1 \# \gamma^*_{ijk} \mbox{ by (\ref{GL})} \\
&= \sum \limits_{j=1}^{n} \pi_1 \# \left( \sum \limits_{k=1}^{n} \gamma^*_{ijk} \right) \\
&= \sum \limits_{j=1}^{n} \pi_1 \# \gamma_{ij}^* \mbox{ by (\ref{13})} \\
&= \mu_i.
\end{align*}
\noindent For identical reasons $\pi_3 \# \sum \limits_{i=1}^{n} \Pi_{ik}= \lambda_k$ and as a consequence:
\begin{align*}
\bs \gamma := (\pi_{1,3} \# \Pi_{ik})_{1 \leqslant i,k \leqslant n} \in \bs \Pi (\bs \mu, \bs \lambda).
\end{align*}
Finally, we have:
\begin{multline*}
W_p(\bs \mu, \bs \lambda)
\leqslant \left( \sum \limits_{ik} \iint d_{ik} (x,z)^p \, \mathrm{d} \gamma_{ik} (x,y,z) \right)^{\frac{1}{p}} \\
= \left(\sum \limits_{ijk} \iiint d_{ik} (x,z)^p \, \mathrm{d} \Pi_{ijk} (x,y,z)\right)^{\frac{1}{p}} \mbox{ by definition of $\bs \gamma$ and $\bs \Pi$}\\
\leqslant \left(\sum \limits_{ijk} \iiint \left( d_{ij} (x,y) + d_{jk}(y,z) \right)^p \, \mathrm{d} \Pi_{ijk} (x,y,z)\right)^{\frac{1}{p}} \mbox{ by (MTI)} \\
\leqslant \left(\sum \limits_{ijk} \iiint d_{ij} (x,y)^p \, \mathrm{d} \Pi_{ijk} (x,y,z)\right)^{\frac{1}{p}} + \left( \sum \limits_{ijk} \iiint d_{jk} (y,z)^p \, \mathrm{d} \Pi_{ijk} (x,y,z)\right)^{\frac{1}{p}} \\
= \left(\sum \limits_{ijk} \iint d_{ij} (x,y)^p \, \mathrm{d} \gamma_{ijk}^* (x,z)\right)^{\frac{1}{p}} + \left(\sum \limits_{ijk} \iint d_{jk} (y,z)^p \, \mathrm{d} \tilde{\gamma}_{ijk}^* (x,y,z)\right)^{\frac{1}{p}} \mbox{ by (\ref{GL})} \\
= \left(\sum \limits_{ij} \iint d_{ij} (x,y)^p \, \mathrm{d} \gamma_{ij}^* (x,y)\right)^{\frac{1}{p}} + \left( \sum \limits_{jk} \iint d_{jk} (y,z)^p \, \mathrm{d} \tilde{\gamma}_{jk}^* (y,z)\right)^{\frac{1}{p}} \\
= W_p (\bs \mu, \bs \nu) +W_p (\bs \nu, \bs \lambda).
\end{multline*}
\end{proof}

\begin{theorem}
Let $X$ be a Polish space. Let $p \in [1,\infty)$. Let $(d_{ij})$ be $n^2$ functions on $X \times X$ valued in $\mathbb{R}_+$ such that:

\textbf{[1]} $\forall (i,j) \in \ldb 1,n \rdb^2, d_{ij}$ is symmetric,

\textbf{[2]} (MTI) is satisfied,

\textbf{[3]} $\forall i \in \ldb 1,n \rdb, \forall x \in X, d_{ii}(x,x) = 0$.

\textbf{[4]} $\forall (i,j) \in \ldb 1,n \rdb^2, i \neq j, \forall (x,y) \in X \times Y, d_{ij}(x,y) \neq 0$.

\noindent $W_p$ is a distance on $\mathcal{P}_p^n(X)$.
\end{theorem}

\begin{ex}
Coming back to the Example \ref{Dirac example}, let us fix all weights
\begin{align}
\bs p = \bs q = \frac{1}{n} \begin{bmatrix}
1 \\
\vdots \\
1
\end{bmatrix}
\end{align}
\noindent and given $x,y \in X^n$, we define the distance $w_p(\bs x,\bs y)$ between them by
\begin{align*}
w_p(\bs x,\bs y) = W_p \left( \frac{1}{n} \begin{bmatrix}
\delta_{x_1} \\
\vdots \\
\delta_{x_n}
\end{bmatrix} , \frac{1}{n} \begin{bmatrix}
\delta_{y_1} \\
\vdots \\
\delta_{y_n}
\end{bmatrix} \right).
\end{align*}
\noindent According to \ref{Dirac example}, $w_p$ is given by
\begin{align*}
w_p ( \bs x, \bs y)^p &= \inf \left\{ \sum_{(i,j) \in \ldb 1,n \rdb^2} t_{ij} d_{ij}(x_i,y_j)^p, \bs t \in \mathcal{M}_n(\mathbb{R}), \sum\limits_{\substack{l=1}}^{n}{t_{il}} = \frac{1}{n}, \sum\limits_{\substack{l=1}}^{n}{t_{lj}} = \frac{1}{n}   \right\} \\
&= \frac{1}{n} \min \left\{ \sum_{(i,j) \in \ldb 1,n \rdb^2} t_{ij} d_{ij}(x_i,y_j)^p, \bs t \mbox{ bistochastic matrix n $\times$ n} \right\},
\end{align*}
\noindent the last equality providing from classical arguments of linear programing. This example show a way to define new distances on a finite product of spaces using $n^2$ distances.
\end{ex}

\section{Conclusion}

The aim of this paper was to present a new point of view in vector-valued optimal transportation. Writing this paper, we discover that in \citep{chen2018vector} that these authors suggest to use the same idea to treat this problem and allowed mixing of species. Their point of view follows a dynamical formulation of optimal transportation (presented in \cite{benamou2000computational}) while in our paper, Kantorovitch's point of view of optimal transportation was our approach angle.

Concerning this approach angle, let us make another small digression about Monge's optimal transportation's problem and present it. Given two probabilities $\mu$ and $\nu$ we are interrested in knowing if the optimal tranference plan between $\mu$ and $\nu$ split mass $i.e$ if (formaly) the support of $\gamma^*$ is included in a function's graph, say $T^*$. A natural question here is if there exists a similar problem associated to $(KP)$? An possible answer is given $\bs \gamma \in\Pi(\bs \mu,\bs \nu)$ we declare that $\bs \gamma$ has a $Monge's$ $form$ if for all $(i,j), \gamma_{ij}$ is included in a function's graph say $T_{ij}$. The main problem here is given $(T_{ij})$, it is not possible to build the associated $\gamma_{ij}$. Indeed, the knowledge of $(T_{ij})$ does not include which parts of $\mu_i$ is transported into $\nu_j$ or in other words we have still to fix $(f_{ij})$ and $(g_{ij})$. This remark makes hard to use only entropic relaxations (see \citep{nenna2016numerical}) to solve our problem  since the main data to find optimal transference plan is to find these $(f_{ij})$ and $(g_{ij})$.

\renewcommand{\abstractname}{Acknowledgements}
\begin{abstract}
The author would like to thank B. Nazaret for his advices and his re-reading of the paper.
\end{abstract}

\bibliographystyle{unsrt}
\bibliography{Bibliomemoire}

\end{document}